\documentclass{amsart}

\usepackage{amssymb}
\usepackage{amsfonts}
\usepackage{geometry}
\usepackage{hyperref}
\usepackage{graphicx}
\usepackage{chngcntr}

\newcommand{\na}{\mathbb{N}}
\newcommand{\re}{\mathbb{R}}
\newcommand{\R}{\mathbb{R}}

\newcommand{\norm}[2]{\|#1\|_{#2}}
\newcommand{\ve}{\varepsilon}

\DeclareMathOperator{\supp}{supp}

\newtheorem{theorem}{Theorem}[section]

\theoremstyle{definition}
\newtheorem{definition}[theorem]{Definition}
\newtheorem{example}[theorem]{Example}

\theoremstyle{remark}
\newtheorem{remark}[theorem]{Remark}

\numberwithin{equation}{section}
\usepackage{enumerate}

\begin{document}

\title[Orlicz Sobolev Inequalities and the Doubling Condition]{Orlicz Sobolev Inequalities and the Doubling Condition}


\author[Korobenko]{Lyudmila Korobenko}
\address{Reed College\\
Portland, Oregon\\
korobenko@reed.edu}


\date{}

\dedicatory{}

\begin{abstract} In \cite{KMR} it has been shown that $(p,q)$ Sobolev inequality with $p>q$ implies the doubling condition on the underlying measure. We show that even weaker Orlicz-Sobolev inequalities, where the gain on the left-hand side is smaller than any power bump, imply doubling. Moreover, we derive a condition on the quantity that should replace the radius on the righ-hand side (which we call `superradius'), that is necessary to ensure that the space can support the Orlicz-Sobolev inequality and simultaneously be non-doubling. 

\end{abstract}

\maketitle

\section{Introduction}\label{sec:intro}

There has been a lot of interest in the theory of Sobolev-type inequalities in metric spaces in the past two decades, and it has seen great developments, see for instance \cite{BB11, Fr1, Fr2, Haj,Haj2, Shan} and references therein. A special interest in
Sobolev-type inequalities arises in the study of regularity of solutions to certain
classes of degenerate elliptic and parabolic PDEs, see for instance \cite[Chapters 7-14]{BB11}, \cite{BM, DJS13, KS, KRSSh, MM} and references therein. In the classical case of an elliptic operator, the Moser or DeGiorgi iteration
technique can be used together with a Sobolev inequality to obtain higher regularity
of solutions \cite{Mos, DeG}. In the subelliptic case, a somewhat similar technique can be used
to obtain H\"{o}lder continuity of weak solutions \cite{FKS82,SW1}. In this case, the Sobolev
inequality on metric balls is used together with a certain accumulating sequence of
Lipschitz cutoff functions to perform the Moser iteration.

In most literature, the doubling condition is central in both the
process of performing the Moser iteration, and in the proof of Sobolev inequality
itself. This condition provides a homogeneous space structure, which makes it
possible to adapt many classical tools available in the Euclidean space \cite{Haj, Shan}. A somewhat surprising result \cite{KMR} says that a classical $(p,q)$ Sobolev inequality with $p>q$ actually implies the doubling condition on the underlying measure.
It is known, however, that there are versions of Sobolev inequality for which the doubling condition
is not necessary. For example, some versions of
Sobolev inequality have been established for lower Ahlfors regular spaces \cite{Miz} under the assumption of a weak Poincar\'{e} type inequality, as well as logarithmic Sobolev inequalities for Gaussian measure \cite{Gro}.

A more general, weaker, version of Sobolev inequality (an Orlicz-Sobolev inequality) has been recently proved for certain non-doubling metric measure spaces, and then successfully applied in the DeGiorgi iteration scheme to prove regularity of solutions to infinitely degenerate elliptic equations \cite{KRSSh}. This poses a question of what types of Sobolev inequalities do imply the doubling condition. In this paper we prove that a Sobolev inequality with a sufficiently large Orlicz bump does also imply the doubling condition. We also derive a lower bound on the quantity that should be put on the right-hand side of the Orlicz-Sobolev inequality in place of the radius, so that it does not imply the doubling condition.

The paper is organized as follows. In Section \ref{main} we give a brief overview of Sobolev spaces on metric spaces, and state our first main result, that a version of Orlicz-Sobolev inequality implies the doubling condition on the measure. Section \ref{sec:proofThm} is dedicated to the proof of this result, Theorem \ref{thm:Sob=>doublingMetric}. In Section \ref{sec:super} we derive a lower bound on the superradius, necessary for the measure to be non-doubling. Finally, the last section contains an explicit calculation of the superradius in the setting of subunit metric spaces.

\section{Orlicz-Sobolev implies doubling}\label{main}
We first briefly review the theory of Sobolev spaces in metric spaces, the reader is referred to \cite[Chapters 1-5]{BB11}, \cite[Sections 1-5]{BM}, \cite[Sections 1-3]{KS} for further details.

Let $(X,d)$ be a metric space. For $y \in X$ and $R> 0$ the $d$-ball centered at $y$ with radius $R$ is defined as $B_d(y,R):=\{x \in X: d(x,y) < R\}$.

Given a function $u: X \rightarrow [-\infty, \infty]$, a non-negative Borel function $g: X\rightarrow [0,\infty]$ is called an \emph{upper gradient} of $u$ if for all curves (i.e. non-constant rectifiable continuous mappings) $\gamma : [0, l_\gamma] \rightarrow X$ it holds
$$
|u(\gamma(0)) - u(\gamma (l_\gamma))| \leq \int_\gamma g \, ds.
$$
In particular, if for some $L \geq 1$, $u: X \rightarrow \re$ is an \emph{$L$-Lipschitz function}, i.e., $|u(x)-u(y)| \leq L \, d(x,y)$ for every $x,y\in X$, then the function $lip(u)$ defined for $x \in X$ as
\begin{equation}\label{deflipu}
lip(u)(x):= \liminf_{r \to 0^+} \sup\limits_{y \in B_d(x,r)} \frac{|u(x)-u(y)|}{r}
\end{equation}
is an upper-gradient for $u$ (see \cite[Proposition 1.14]{BB11}). Notice also that $lip(u)(x) \leq L$ for every $x \in X$.

Let $\mu$ be a Borel measure on $(X,d)$ such that $0 < \mu(B) < \infty$ for every $d$-ball $B \subset X$. For $1 \leq p < \infty$ and $u \in L^p(X, \mu)$ set
$$
\norm{u}{N^{1,p}}^p:= \int_X |u|^p d\mu + \inf\limits_{g} \int_X g^p d\mu,
$$
where the infimum is taken over all the upper-gradients $g$ of $u$. Given a $d$-ball $B \subset X$, its Newtonian space with zero boundary values is defined as
$$
N_0^{1,p}(B):=\{f|_B : \norm{f}{N^{1,p}} < \infty \text{ and } f\equiv 0 \text{ on } X \setminus B\}.
$$

First, we recall a classical Sobolev inequality

\begin{definition}
	Let $(X, d, \mu)$ be as above. Given $1\leq p < \infty$ and $1< \sigma <\infty$, we say that the triple $(X, d, \mu)$ admits a \emph{weak $(p\sigma, p)$-Sobolev inequality} with a (finite) constant $C_S > 0$ if for every $d$-ball $B:=B_d(y,r) \subset X$ and every function $w\in N_0^{1,p}(B)$ it holds true that
	\begin{align}\label{sob_classic}
	\left\Vert w\right\Vert _{L^{p\sigma }\left( \mu_{B}\right) }\leq C
	r\left( B\right) \left\Vert g\right\Vert _{L^{p}\left(\mu_{B}\right)}
	\end{align}
	for all upper-gradients $g$ of $w$.
\end{definition}
In the above definition and everywhere below we use the notation 
$$
d\mu_{B}:=\frac{d\mu}{\mu(B)}.
$$ 
It has been proved in \cite{KMR} that the above Sobolev inequality implies that the measure $\mu$ is doubling on $(X, d)$. We are now interested in the question whether some weaker versions of Sobolev inequality also imply the doubling property of the measure. We first show by a simple counter example, that $(p,p)$-Sobolev inequality does not necessarily imply doubling.
\begin{example}
	In $\re^2$ consider the operator $L$ defined as $L=div A\nabla$ with $A=A(x,y)=diag\{1,e^{1/|x|}\}$, $(x,y)\in\re^2$. Let $d$ be a subunit metric associated to $L$ (see \cite{KRSSh} for definitions), and $\mu=|\  |$ the Lebesgue measure. Then it was shown in \cite{KRSSh} that the measure of a metric ball centered at the origin satisfies the following estimate
	\[
	|B(0,r)|\approx r^4 e^{-\frac{1}{r}}
	\]
	which is easily seen to be non-doubling. On the other hand, if $w\in W^{1,p}_{0}(B)$ then
	\[
	|w(x,y)|=\left\vert\int_{-\infty}^{x}\frac{\partial}{\partial t}w(t,y)dt\right\vert \leq \int_{-\infty}^{\infty}\left\vert\frac{\partial}{\partial t}w(t,y)\right\vert dt\leq \int_{-\infty}^{\infty}\left\vert\nabla w(t,y)\right\vert dt
	\]
	Raising to the power $p$ and integrating we obtain
	\[
	\iint_{B}|w|^p d\mu\leq Cr^{p}\iint_{B}|\nabla w|^p d\mu
	\]
	where we used H\"{o}lder inequality and the fact that $w(x,y) =0$ for $|x|>r$. 
	\end{example}

We now look for a stronger form of Sobolev inequality weaker than (\ref{sob_classic}) that would still imply the doubling property of the measure. Roughly speaking we want to put a ``bump'' in the norm on the left-hand side, smaller than any power bump. One natural class of function spaces to consider is the class of Orlicz spaces. We now give the relevant definitions. Suppose that $\mu $ is a $%
\sigma $-finite measure on a set $X$, and $\Phi :\left[ 0,\infty \right)
\rightarrow \left[ 0,\infty \right) $ is a Young function, which for our
purposes is a convex piecewise differentiable (meaning there are at most
finitely many points where the derivative of $\Phi $ may fail to exist, but
right and left hand derivatives exist everywhere) function such that $\Phi
\left( 0\right) =0$ and

\begin{equation*}
\frac{\Phi \left( x\right) }{x}\rightarrow \infty \text{ as }x\rightarrow
\infty, \quad\text{and}\quad
\frac{\Phi \left( x\right) }{x}\rightarrow 0 \text{ as }x\rightarrow
0.
\end{equation*}%
We also note here that from the assumptions on $\Phi$ it follows that the function
\begin{equation*}
\Psi(x):=\frac{\Phi(x)}{x}
\end{equation*}
is increasing and satisfies
\begin{equation*}
\Psi(x)\rightarrow \infty \text{ as }x\rightarrow
\infty, \quad\text{and}\quad
\Psi(x)\rightarrow 0 \text{ as }x\rightarrow
0.
\end{equation*}%
Let $L_{\ast }^{\Phi }$ be the set of measurable functions $f:X\rightarrow 
\mathbb{R}$ such that the integral%
\begin{equation*}
\int_{X}\Phi \left( \left\vert f\right\vert \right) d\mu ,
\end{equation*}%
is finite, where as usual, functions that agree almost everywhere are
identified. Since the set $L_{\ast }^{\Phi }$ may not be closed under scalar
multiplication, we define $L^{\Phi }$ to be the linear span of $L_{\ast
}^{\Phi }$, and then define%
\begin{equation}\label{norm_def}
\left\Vert f\right\Vert _{L^{\Phi }\left( \mu \right) }\equiv \inf \left\{
k\in \left( 0,\infty \right) :\int_{X}\Phi \left( \frac{\left\vert
	f\right\vert }{k}\right) d\mu \leq 1\right\} .
\end{equation}%
The Banach space $L^{\Phi }\left( \mu \right) $ is precisely the space of
measurable functions $f$ for which the norm $\left\Vert f\right\Vert
_{L^{\Phi }\left( \mu \right) }$ is finite. 
\begin{definition}
Let $(X, d, \mu)$ be as above. Given a Young function $\Phi$, we say that the triple $(X, d, \mu)$ admits an \emph{Orlicz-Sobolev inequality} with a $\Phi$ bump and a (finite) constant $C_S > 0$ if for every $d$-ball $B:=B_d(y,r) \subset X$ and every function $w\in N_0^{1,1}(B)$ it holds true that
\begin{equation}\label{MetricSob}
\left\Vert w\right\Vert _{L^{\Phi }\left( \mu_{B}\right) }\leq C
r\left( B\right) \left\Vert g\right\Vert _{L^{1}\left(\mu_{B}\right)},\ \ \ \ \ \supp w\subset B,
\end{equation}
for all upper-gradients $g$ of $w$.
\end{definition}

Here is our first main result

\begin{theorem}\label{thm:Sob=>doublingMetric} Suppose that the triple $(X, d, \mu)$ admits an Orlicz-Sobolev inequality (\ref{MetricSob}). Assume also that the function $\Phi$ satisfies
	\begin{equation}\label{phi_cond1}
	\Phi(t)\geq t\left(\ln t\right)^{\alpha},\quad \forall t>1,
	\end{equation}
	for some $\alpha>1$.
Then, the measure $\mu$ is doubling on $(X, d)$. More precisely, there exists a constant $C_D \geq 1$, depending only on $\alpha$ and $C_S$, such that
\begin{equation*}
\mu(B_d(y,2r)) \leq C_D \, \mu(B_d(y,r)) \quad \forall y \in X, r >0.
\end{equation*}
\end{theorem}


\section{Proof of Theorem \ref{thm:Sob=>doublingMetric}}\label{sec:proofThm}

Given a $d$-ball $B:=B_d(y,r)$ set $B^*:=B_d(y,2r)$ and define a family of $d$-Lipschitz functions $\{\psi_j\}_{j \in \na} \subset N_0^{1,1}(B) \subset N_0^{1,1}(B^*)$ as follows: for $j \in \na$ set $r_{1}=r$, $\lim_{j\rightarrow \infty }r_{j}=\frac{1}{2}r$, $r_{j}-r_{j+1}=\frac{c}{j^{\gamma }}r$ for a uniquely determined constant $c$ and $\gamma >1$
 and
\begin{equation}\label{defpsijMetric}
\psi_j(x):= \left(\frac{r_j - d(x,y)}{r_j - r_{j+1}} \right)^+ \wedge 1.
\end{equation}
Also for $j \in \na$, define the $d$-balls $B_j$ as
$$
\frac{1}{2}B \subset B_j:= \{x \in X: d(x,y) \leq r_j\} \subset B \subset B^*=2B.
$$
Our first step will be to apply the Orlicz-Sobolev inequality \eqref{MetricSob} to $\psi_j$ on $B^*$ by choosing the upper-gradient $g_j:=lip(\psi_j)$ as defined in \eqref{deflipu}. In particular, it follows that
\begin{equation}\label{condpsij}
g_j(x) \leq \frac{1}{c}\frac{j^{\gamma}}{r} \chi_{B_j}(x) \quad \text{and} \quad 0 \leq \psi_j(x) \leq 1 \quad \forall x \in X.
\end{equation}
Then, by the Orlicz-Sobolev inequality \eqref{MetricSob}  applied to each $\psi_j$ on $B^*=2B$ (and using the fact that each $\psi_j$ is supported in $B_j$, so that $B^*$ can be replaced by $B_j$ in the integrals), we obtain
\begin{equation}\label{Sobappliedtopsij}
\left\Vert \psi_j\right\Vert _{L^{\Phi }\left( \mu_{B^*}\right) }\leq C_Sr\left( B\right) \left\Vert g\right\Vert _{L^{1}\left(\mu_{B^*}\right)}\\
=2C_Sr\int_{B_j}g_j\frac{d\mu}{\mu(B^*)}\leq \frac{2C_S}{c}j^{\gamma}\frac{\mu(B_j)}{\mu(B^*)}
\end{equation}
where in the last inequality we used (\ref{condpsij}). We now denote $\widetilde{C}_S=2C_S/c$ and note that this constant depends only on $C_S$ and $\gamma$. Using definition (\ref{norm_def}) the norm bound (\ref{condpsij}) implies
$$
\int\Phi\left(\frac{\psi_j}{\widetilde{C}_Sj^{\gamma}\frac{\mu(B_j)}{\mu(B^*)}}\right)\frac{d\mu}{\mu(B^*)}\leq 1.
$$
From (\ref{defpsijMetric}) we then have $\psi_j=1$ on $B_{j+1}$ and thus
\begin{equation}\label{iter_setup}
\Phi\left(\frac{1}{\widetilde{C}_Sj^{\gamma}\frac{\mu(B_j)}{\mu(B^*)}}\right)\leq \frac{\mu(B^*)}{\mu(B_{j+1})}.
\end{equation}
Using the notation
\[
P_j:=\frac{\mu(B^*)}{\widetilde{C}_Sj^{\gamma}\mu(B_j)},
\]
and condition (\ref{phi_cond1}) we obtain from (\ref{iter_setup})
\[
P_{j+1}\geq P_{j}\frac{(\ln P_{j})^{\alpha}}{\widetilde{C}_S(j+1)^{\gamma}}.
\]
We now show by induction that $P_j\geq P_1 e^{j-1}$ $\forall j\geq 1$, provided $P_1$ is sufficiently large depending on $\widetilde{C}_S$. The base case $j=1$ is trivially true. Assume $P_j\geq P_1 e^{j-1}$ for some $j\geq 1$, then
\[
P_{j+1}\geq P_1 e^{j}\frac{(\ln P_1+j-1)^{\alpha}}{e \widetilde{C}_S(j+1)^{\gamma}}.
\]
Since $\alpha>1$ we can choose $1<\gamma<\alpha$ so there holds
$$
\frac{(\ln P_1+j-1)^{\alpha}}{e \widetilde{C}_S(j+1)^{\gamma}}\to \infty\quad\text{as } j\to \infty.
$$
Thus we can arrange 
\[
\frac{(\ln P_1+j-1)^{\alpha}}{e \widetilde{C}_S(j+1)^{\gamma}}\geq 1\  \  \forall j\geq 1
\]
by choosing $P_1\geq C(\widetilde{C}_S,\alpha,\gamma)$. This would give $P_j\to\infty$ as $j\to \infty$, which is a contradiction since
\[
P_j=\frac{\mu(B^*)}{\widetilde{C}_Sj^{\gamma}\mu(B_j)}\leq \frac{\mu(B^*)}{\widetilde{C}_Sj^{\gamma}\mu(1/2B)}\to 0\quad\text{as } j\to \infty.
\]
Therefore, it must be that
\[
P_1\leq C(\widetilde{C}_S,\alpha,\gamma) \implies \frac{\mu(B^*)}{\mu(B_1)}= \frac{\mu(2B)}{\mu(B)}\leq C,
\]
so the measure $\mu$ is doubling.

\begin{remark}
	Note that we cannot conclude the doubling property without imposing the condition $\Phi(t)\geq t(\ln t)^{\alpha}$, $\alpha>1$. In particular, non-doubling measure metric spaces may support Orlicz-Sobolev inequalities with ``weak enough'' bumps, e.g. log-Sobolev inequalities with Gaussian measure \cite{Gro} which correspond to $\Phi(t)= t\ln t$.
\end{remark}
\section{Superradius estimates}\label{sec:super}
We now introduce a weaker version of Orlicz-Sobolev inequality

\begin{definition}
	Let $(X, d, \mu)$ be as before. Given a Young function $\Phi$, we say that the triple $(X, d, \mu)$ admits a \emph{weak Orlicz-Sobolev inequality} with superradius $\varphi(r)$ and a (finite) constant $C_S > 0$ if for every $d$-ball $B:=B_d(y,r) \subset X$ and every function $w\in N_0^{1,1}(B)$ it holds true that
	\begin{equation}\label{weakMetricSob}
	\left\Vert w\right\Vert _{L^{\Phi }\left( \mu_{B}\right) }\leq C_S
	\varphi(r)\left( B\right) \left\Vert g\right\Vert _{L^{1}\left(\mu_{B}\right)},\ \ \ \ \ \supp w\subset B,
	\end{equation}
	for all upper-gradients $g$ of $w$. Here $\varphi(r)$ is nondecreasing and $\varphi(r)\geq r$.
\end{definition}
We would now like to derive an estimate on the superradius $\varphi(r)$ that is necessary for the space $(X, d, \mu)$ to simultaneously support a weak Orlicz-Sobolev inequality (\ref{weakMetricSob}) and allow the measure to be non-doubling.
\begin{theorem}\label{thm:Sob=>superradius} Suppose that the triple $(X, d, \mu)$ admits a weak Orlicz-Sobolev inequality (\ref{weakMetricSob}) with superradius $\varphi(r)$. Assume also that the function $\Phi$ satisfies
	\begin{equation}\label{phi_cond}
	\Phi(t)\geq t\left(\ln t\right)^{\alpha},\quad \forall t>1,
	\end{equation}
	for some $\alpha>1$.
	Then for every $\ve>0$ the superradius $\varphi$ satisfies
	\begin{equation}\label{super_est}
	\frac{\varphi(r(B))}{r(B)}\geq C_\ve\left(\ln\left[\frac{\mu(B)}{\mu\left(\frac{1}{2}B\right)}\right]\right)^{\alpha-1-\ve},
	\end{equation}
	where the constant $C_{\ve}>0$ depends on $\ve$, $\alpha$ and $C_S$.
	In particular, if $\varphi(r)=r$, the measure $\mu$ is doubling on $(X, d)$.
\end{theorem}
\begin{proof}
	Proceeding as in the proof of Theorem \ref{thm:Sob=>doublingMetric} we arrive at
	\begin{equation}\label{ineq_super}
	\Phi\left(\frac{1}{\widetilde{C}_S(r)j^{\gamma}\frac{\mu(B_j)}{\mu(B^*)}}\right)\leq \frac{\mu(B^*)}{\mu(B_{j+1})},
	\end{equation}
	where now 
	$$
	\widetilde{C}_S(r)=\frac{2C_S}{c}\frac{\varphi(2r)}{2r}.
	$$
The inequality we have for $P_j=\frac{\mu(B^*)}{\widetilde{C}_S(r)j^{\gamma}\mu(B_j)}$	is
\[
P_{j+1}\geq P_{j}\frac{(\ln P_{j})^{\alpha}}{\widetilde{C}_S(r)(j+1)^{\gamma}}.
\] 
The induction assumption $P_j\geq P_1e^{j-1}$ gives
\[
P_{j+1}\geq P_1 e^{j}\frac{(\ln P_1+j-1)^{\alpha}}{e \widetilde{C}_S(r)(j+1)^{\gamma}}.
\]
We now derive a condition on $P_1$ that will guarantee that $P_{j+1}\geq P_1 e^j$. We need
\begin{equation}\label{p1-cond}
\ln P_1+j-1\geq e^{1/\alpha} \widetilde{C}_S(r)^{1/\alpha}(j+1)^{\gamma/\alpha}.
\end{equation}
Choosing $1<\gamma<\alpha$ and using Young's inequality with $p=\frac{\alpha}{\gamma}$, $p'=\frac{\alpha}{\alpha-\gamma}$ we have
\[
e^{1/\alpha} \widetilde{C}_S(r)^{1/\alpha}(j+1)^{\gamma/\alpha}\leq \frac{\alpha-\gamma}{\alpha}(e\widetilde{C}_S(r))^{\frac{1}{\alpha-\gamma}}+\frac{\gamma}{\alpha}(j+1).
\]
To satisfy (\ref{p1-cond}) it is then sufficient to require
\[
\ln P_1\geq 2+ \frac{\alpha-\gamma}{\alpha}(e\widetilde{C}_S(r))^{\frac{1}{\alpha-\gamma}},
\]
and we then arrive at a required contradiction $P_j\to\infty$ as $j\to \infty$. The condition
\[
\widetilde{C}_S(r)\geq C(\alpha,\gamma)(\ln P_1)^{\alpha-\gamma}
\]
is thus necessary and is guaranteed by 
\[
\frac{\varphi(2r)}{2r}\geq C(C_S,\alpha,\gamma)\left(\ln\left[\frac{\mu(2B)}{\mu\left(B\right)}\right]\right)^{\alpha-\gamma}.
\]
Since $\gamma$ can be chosen arbitrarily close to $1$, this concludes (\ref{super_est}).

\end{proof}

\section{Example}
In this section we restate Theorem \ref{thm:Sob=>doublingMetric} in the setting of a metric measure space related to (degenerate) elliptic operators. We then consider a concrete example of Orlicz-Sobolev inequality in such a space, and calculate the quantities in (\ref{super_est}). This example suggests that the exponent on the right hand side of (\ref{super_est}) is not sharp and could be improved to $\alpha$.
\subsection{Subelliptic version of main theorem}
We follow the terminology and notation of \cite{KMR} (see also \cite[Section 1]{SW1}). Consider an open subset $\Omega\subset\mathbb{R}^n$ (in the Euclidean topology) and let
$$
Q : \Omega \rightarrow \{\text{non-negative semi-definite }  n \times n \text{ matrices}\}
$$
be a locally bounded function on $\Omega$. For a Lipschitz function $u : \Omega \rightarrow \re$ (throughout this subsection, Lipschitz means Lipschitz with respect to the Euclidean distance), define its $Q$-gradient Lebesgue-a.e. in $\Omega$ as
$$
[\nabla u]_Q:= ( \nabla u^T Q \nabla u)^{\frac{1}{2}}.
$$
Let $d$ be any metric on $\R^n$.
\begin{definition}[Standard sequence of accumulating Lipschitz functions]
	\label{def_cutoff}Let $\Omega $ be a bounded domain in $\mathbb{R}^{n}$. Fix 
	$r>0$ and $x\in \Omega $. We define a $\left(Q,d\right) $-\emph{%
		standard} sequence of Lipschitz cutoff functions $\left\{ \psi _{j}\right\}
	_{j=1}^{\infty }$ at $\left( x,r\right) $, along with sets $B(x,r_{j})\supset \mathrm{supp}(\psi _{j})$, to be a sequence satisfying $%
	\psi _{j}=1$ on $B(x,r_{j+1})$, $r_{1}=r$, ${r_{\infty }\equiv
	\lim_{j\rightarrow \infty }r_{j}=\frac{1}{2}r}$, $r_{j}-r_{j+1}=\frac{c}{%
		j^{\gamma }}r$ for a uniquely determined constant $c$ and $\gamma >1$, and ${
	\left\Vert [\nabla\psi_j]_Q\right\Vert _{\infty }\leq K\frac{%
		j^{\gamma }}{r}}$.
\end{definition}

For $1 \leq p < \infty$, let $\mathcal{W}^{1,p}_Q(\Omega, dx)$ denote the closure of the Lipschitz functions on $\Omega$ under the norm
$$
\norm{u}{\mathcal{W}^{1,p}_Q(\Omega, dx)}:= \norm{u}{L^p(\Omega, dx)} + \norm{[\nabla u]_Q}{L^p(\Omega,dx)}.
$$
We say that  $(\Omega, d, Q)$ admits an \emph{Orlicz-Sobolev inequality} with a (finite) constant $C_S > 0$ if for every $d$-ball $B:=B_d(y,r) \subset X$, with $0 < r < \text{dist}(y, \partial \Omega)/2$, and every function $w \in \mathcal{W}^{1,1}_Q(\Omega, dx)$ with $\text{supp}(w) \subset B$ it holds true that
\begin{align}\label{weakSubellipticSob}
\left\Vert w\right\Vert _{L^{\Phi }\left( \mu_{B}\right) }\leq C_S
r\left( B\right) \left\Vert [\nabla w]_Q\right\Vert _{L^{1}\left(\mu_{B}\right)},\ \ \ \ \ \supp w\subset B,
\end{align}
where $d\mu_B$ now stands for $dx/|B|$, and $|B|$ is the Lebesgue measure of $B$. 
We have the following version of Theorem \ref{thm:Sob=>doublingMetric}

\begin{theorem}\label{thm:Sob=>doublingSubelliptic} Suppose that the structure $(\Omega, d, Q)$ admits accumulating sequences of Lipschitz cut-off functions as well as an Orlicz-Sobolev inequality (\ref{weakSubellipticSob}). 
	Assume also that the function $\Phi$ satisfies
	\begin{equation*}
	\Phi(t)\geq t\left(\ln t\right)^{\alpha},\quad \forall t>1,
	\end{equation*}
	for some $\alpha>\gamma$.
	Then, the Lebesgue measure is doubling on $(\Omega, d)$. More precisely, there exists a constant $C_D \geq 1$, depending only on $\alpha$, $C_S$, and $\gamma$ and $K$ from Definition \ref{def_cutoff} such that
	\begin{equation*}
	|B_d(y,2r)| \leq C_D \, |B_d(y,r)| \quad \forall y \in X, r >0.
	\end{equation*}
\end{theorem}

\begin{proof} Given a $d$-ball $B:=B_d(y,r)$, with  $0 < r < \text{dist}(y, \partial \Omega)/2$, just as in the proof of Theorem \ref{thm:Sob=>doublingMetric}, apply the weak-Sobolev inequality \eqref{weakSubellipticSob} to the accumulating sequence of Lipschitz cut-off functions $\{\psi_j\}$ on the $d$-ball $B^*:=B_d(y,2r)$ and, for $j \in \na$, set $B_j:=\mathrm{supp}(\psi_j)$, to obtain
	\begin{equation*}
	\left\Vert \psi_j\right\Vert _{L^{\Phi }\left( \mu_{B^*}\right) }\leq C_Sr\left( B\right) \left\Vert [\nabla\psi_j]_Q\right\Vert _{L^{1}\left(\mu_{B^*}\right)}\\
	\leq C_S\left\Vert [\nabla\psi_j]_Q\right\Vert _{\infty }\left\Vert 1\right\Vert _{L^{1}\left(\mu_{B^*}\right)}\leq C_SKj^{\gamma}\frac{\mu(B_j)}{\mu(B^*)}.
	\end{equation*}
	Note that this is precisely (\ref{Sobappliedtopsij}) with $C_SK$ in place of $C_S/c$. With the notation $\widetilde{C}_S=C_SK$ the rest of the proof repeats verbatim the proof of Theorem \ref{thm:Sob=>doublingMetric}.
	\end{proof}

Finally, we can also restate Theorem \ref{thm:Sob=>superradius} and the proof will follow similarly
\begin{theorem}\label{thm:Sob2=>superradius} Suppose that the structure $(\Omega, d, Q)$ admits accumulating sequences of Lipschitz cut-off functions as well as an Orlicz-Sobolev inequality with superradius $\varphi(r)$
	\begin{align*}
	\left\Vert w\right\Vert _{L^{\Phi }\left( \mu_{B}\right) }\leq C_S
	\varphi(r\left( B\right))\left\Vert [\nabla w]_Q\right\Vert _{L^{1}\left(\mu_{B}\right)},\ \ \ \ \ \supp w\subset B.
	\end{align*}
	Assume also that the function $\Phi$ satisfies
	\begin{equation*}
	\Phi(t)\geq t\left(\ln t\right)^{\alpha},\quad \forall t>1,
	\end{equation*}
	for some $\alpha>\gamma$.
Then for every $\ve>0$ the superradius $\varphi$ satisfies
\begin{equation}\label{superrad}
	\frac{\varphi(r(B))}{r(B)}\geq C_\ve\left(\ln\left[\frac{|2B|}{|B|}\right]\right)^{\alpha-1-\ve},
\end{equation}
	where the constant $C_{\ve}>0$ depends on $\ve$, $\alpha$ and $C_S$.
	In particular, if $\varphi(r)=r$, the Lebesgue measure is doubling on $(\Omega, d)$.
\end{theorem}

 \subsection{Example}
The following example suggests that the estimate (\ref{superrad}) might not be sharp. More precisely, there exists a matrix $Q$, metric $d$, and a subset $\Omega\subseteq \R^n$ such that 
\begin{equation}\label{OS}
\left\Vert w\right\Vert _{L^{\Phi }\left( \mu_{B}\right) }\leq C
\varphi(r\left( B\right))\left\Vert [\nabla w]_Q\right\Vert _{L^{1}\left(\mu_{B}\right)},
\end{equation}
for all $w \in \mathcal{W}^{1,1}_Q(\Omega, dx)$ with $\text{supp}(w) \subset B$, where
\begin{equation}\label{Phi}
\Phi(t)=t(\ln t)^{\alpha},\  \alpha>1, \quad \forall t>1
\end{equation}
and $\varphi(r)$ satisfies
\begin{equation}\label{superrad_upper}
	\frac{\varphi(r(B))}{r(B)}\approx\left(\ln\left[\frac{|2B|}{|B|}\right]\right)^{\alpha}.
\end{equation}
Therefore, we expect the same lower bound on the superradius might be necessary, i.e. it should be possible to improve (\ref{superrad}) to 
\[
\frac{\varphi(r(B))}{r(B)}\geq C\left(\ln\left[\frac{|2B|}{|B|}\right]\right)^{\alpha}.
\]
Estimate (\ref{superrad_upper}) is a consequence of \cite[Proposition 80]{KRSSh}. More precisely, let $n=2$, $Q(x,y)=diag\{1,f^{2}(x)\}$ where $f(x)=exp(-1/|x|^{\sigma})$, $0<\sigma<1$, and let $d$ be the metric subunit to $Q$, see \cite{SW1, KRSSh} for the definition. Then the function $F(x):=-\ln f(x)=1/|x|^{\sigma}$ satisfies the condition of \cite[Proposition 80]{KRSSh} provided $\sigma \alpha<1$. Proposition 80 then says that Orlicz-Sobolev inequality (\ref{OS}) with $\Phi$ satisfying (\ref{Phi}) holds in the ball $B=B(0,r)$ with 
$$
\varphi(r)=C|F'(r)|^{\alpha}r^{\alpha+1}
$$
provided $\lim_{r\to 0}\varphi(r)=0$. Now, if $F(x)=1/|x|^{\sigma}$ and $\sigma \alpha<1$ we have
$$
\varphi(r)=C|F'(r)|^{\alpha}r^{\alpha+1}=C\frac{r^{\alpha+1}}{r^{\alpha(\sigma+1)}}=C\frac{r}{r^{\alpha\sigma}}\to 0,\quad\text{as}\  r\to 0.
$$
We therefore only need to check estimate (\ref{superrad_upper}). Using the estimate from \cite[Conclusion 45]{KRSSh} we have
$$
|B(0,r)|\approx \frac{f(r)}{|F'(r)|^{2}}\approx r^{2(\sigma+1)}e^{-\frac{1}{r^{\sigma}}},
$$
and therefore 
$$
\left(\ln\left[\frac{|2B|}{|B|}\right]\right)^{\alpha}\approx \frac{1}{r^{\sigma \alpha}}\approx \frac{\varphi(r)}{r}.
$$

\end{document}